\documentclass[a4,11pt]{article}
\usepackage{amssymb,amsfonts,amsthm,amsmath}
\usepackage{epsfig}
\usepackage[utf8]{inputenc}
\usepackage{mathrsfs}

\newtheorem{definition}{Definition}
\newtheorem{property}{Property}
\newtheorem{theorem}{Theorem}
\newtheorem{lemma}{Lemma}
\usepackage{dsfont}
\usepackage{mathtools}

\begin{document}

\title{ \Large{\bf Properties of fractional integral operators involving the
three-parameters Mittag-Leffler function in the kernels with respect to another function}}


\author{ 
  {\bf {\large{D. S. Oliveira}}} \\
 {\small Coordination of Civil Engineering, UTFPR, } \\
{\small Guarapuava, PR, 85053-525, Brazil} \\
 {\small oliveiradaniela@utfpr.edu.br} \\ \\
  }

 \date{}

\maketitle

\thispagestyle{empty}
\vspace{-.4cm}


\noindent{\bf Abstract:} {This paper aims to investigate properties associated with fractional integral operators 
involving the three-parameters Mittag-Leffler function in the kernels with respect to another function. 
We prove that the Cauchy problem and the Volterra integral equation are equivalent. We find 
a closed-form to the solution of the Cauchy problem using successive approximations method 
and $\psi$-Caputo fractional derivative.}


\vspace*{.2cm}
\section{Introduction} 

In the last years, the number of integral and differential operators has increased a lot 
\cite{Almeida,Oliveira,Sales,Vanterler,Yang}. Some of these operators contain in its kernels 
the so-called special functions, for example: hypergeometric function \cite{Grinko}, Meijer 
G-function, Fox H-function \cite{Sri} and three-parameters Mittag-Leffler function.
In 1971, Prabhakar introduced in the kernel of the Riemann-Liouville fractional integral 
the three-parameters Mittag-Leffler function, \cite{Prabhakar}. 
In 2002, Kilbas et al. investigated an integrodifferential equation 
involving the Riemann-Liouville fractional derivative and the fractional integral operator 
developed by Prabhakar, \cite{Kilbas2002}. In 2004, the same authors, proved some properties 
associated with the generalized operator defined by Prabhakar, \cite{Kilbas2004}. Srivastava 
and Tomovski, in 2009, proposed the fractional integral operator, which contains in its kernel 
the generalized Mittag-Leffler function, \cite{Srivastava}. 
Garra et al., in 2014, defined the Hilfer-Prabhakar fractional derivative, this fractional 
derivative is a generalization of Hilfer derivative in which Riemann-Liouville fractional 
integral is replaced by Prabhakar fractional integral, \cite{Garra}. In 2016, Dorrego 
introduced the $k$-Mittag-Leffler function in the kernel of $k$-Riemann-Liouville fractional 
integral, \cite{Dorrego}. Recently, in 2018, Sousa and Capelas de Oliveira defined fractional 
integral operators containing in their kernel the Mittag-Leffler function \cite{Vanterler2} and, 
Baleanu and Fernandez introduced a fractional derivative involving this same function in its kernel 
\cite{Fernandez}. Atangana and Baleanu proposed the so-called AB fractional derivative operators which
contain in the kernel the one-parameter Mittag-Leffler function, \cite{Atangana}.
Based in these fractional operators it was proposed fractional integral 
operators which contain in its kernel the Mittag-Leffler function with respect to another 
function, \cite{Yang}.

Numerous applications have emerged from these operators, among which we can mention:
Zhao and Sun applied the Caputo type fractional derivative with a Prabhakar-like kernel 
to discuss the anomalous relaxation model and its solution, \cite{Zhao} and G\'{o}rska
et al. published a note about the work of Zhao and Sun, \cite{Horzela}; Garra and Garrappa
used fractional operators containing in its kernels the Prabhakar function to present their 
applications in dielectric models of Havriliak-Negami type, \cite{Garrappa}; Giusti
discussed some generalities of relaxation processes involving Prabhakar derivatives, 
\cite{Giusti}; Sandev presented analytical results related to the generalized Langevin 
equation with regularized Prabhakar derivative operator, \cite{Sandev}; Yavuz et al. used
AB fractional derivative operators to solved time-fractional partial differential equations,
\cite{Yavuz}. Xiao et al. discussed in their book applications of fractional derivatives 
with nonsingular kernels in viscoelasticity, \cite{Xiao}.

The main objective of this work is to prove some properties associated with the fractional 
operator proposed by Yang and use it in a fractional differential equation.
The structure of the paper is as follows: In Section \ref{Sec:2}, we present the space of 
functions used throughout the text and some basic definitions and properties associated 
with $\psi$-Riemann-Liouville fractional integral and $\psi$-Caputo fractional derivative.
Section \ref{Sec:3} is devoted to the study of some theorems and lemmas related to
fractional integral operators involving the three-parameters Mittag-Leffler function in 
the kernels with respect to another function. In Section \ref{Sec:4}, we discuss the equivalence
between the Cauchy problem and the Volterra integral equation. We find a closed-form to the 
solution of the Cauchy problem using successive approximations method and $\psi$-Caputo 
fractional derivative of order $\beta$, where $n-1<\beta<n,\,\,n\in\mathbb{N}$,
subject to the initial conditions. Finally, in Section \ref{Sec:5}, we define,
inverse operators of fractional integral operators involving the three-parameters 
Mittag-Leffler function in the kernels with respect to another function. Concluding 
remarks close the paper.

\section{Preliminaries}
\label{Sec:2}
In this section, we present definitions of weighted and continuous spaces of functions \cite{Kilbasbook}.
This section contains, also, definitions and properties associated with the $\psi$-Riemann-Liouville 
fractional integral and the $\psi$-Caputo fractional derivative, \cite{Almeida}. 

Let $\Omega=[a,b]\,\,(0<a<b<\infty)$ be a finite interval of the real axis $\mathbb{R}$
and $n\in\mathbb{N}_0=\{0,1,\ldots\}$. We denote by $C^n(\Omega)$ a space of functions
which are $n$ times continuously differentiable on $\Omega$ with the norm
\begin{eqnarray*}
{\lVert f\rVert}_{C^{n}(\Omega)}=\sum_{k=0}^{n}{\lVert f^{(k)}\rVert}_{C(\Omega)}=
\sum_{k=0}^{n}\max_{x\in\Omega}|f^{(k)}(x)|, \quad\quad n\in\mathbb{N}_0.
\end{eqnarray*}
In particular, for $n=0$, $C^{0}(\Omega)=C(\Omega)$ is the space of the continuous function 
$f$ on $\Omega$ with the norm defined by 
\begin{eqnarray*}
{\lVert f\rVert}_{C(\Omega)}=\max_{x\in\Omega}|f(x)|.
\end{eqnarray*}
The weighted space $C_{\nu,\psi}[a,b]$ of functions $f$ given on $(a,b]$ 
with $\nu\in\mathbb{R}\,\,(0\leq{\nu}<1)$ is
\begin{eqnarray*}
C_{\nu,\psi}(\Omega)=\{f:(a,b]\rightarrow\mathbb{R}; (\psi(x)-\psi(a))^{\nu}f(x)\in{C(\Omega)}\},
\end{eqnarray*}
with the norm
\begin{eqnarray}
{\lVert f\rVert}_{C_{\nu,\psi}(\Omega)}={\lVert (\psi(x)-\psi(a))^{\nu}f(x)\rVert}_{C(\Omega)}=
\max_{x\in\Omega}|(\psi(x)-\psi(a))^{\nu}f(x)|. \label{norm}
\end{eqnarray}
If $\nu=0$, we have $C_{0,\psi}(\Omega)=C(\Omega)$.\\
\begin{definition}\textnormal{\cite{Prabhakar}}
Let $\alpha,\gamma,\rho,z\in\mathbb{R}$ with $\rho>0$. The three-parameters Mittag-Leffler function
is given by
\begin{eqnarray}
E_{\rho,\alpha}^{\gamma}(z)=\sum_{k=0}^{\infty}\frac{(\gamma)_k}{\Gamma(\rho{k}+\alpha)}\frac{z^k}{k!},
\end{eqnarray}
where
$(\gamma)_k$ is the Pochhammer symbol defined as follow
\begin{eqnarray}
(\gamma)_{k}=
\left\{
\begin{array}{l l}
1, &\textnormal{for} \quad k=0\\
\gamma(\gamma+1)\cdots (\gamma+k-1), &\textnormal{for} \quad k=1,2,\ldots, \label{k-pochhammer}\\
\end{array}
\right.
\end{eqnarray}
or, in terms of a quotient of gamma functions, 
\begin{eqnarray}
(\gamma)_{k}=\frac{\Gamma(\gamma+k)}{\Gamma(\gamma)}.
\end{eqnarray}
\end{definition}
\begin{definition}
\label{beta}
Let $x,y\in\mathbb{R}$ with $x>0$ and $y>0$. The beta function, $B(x,y)$, is defined by the Euler 
integral of the first kind
\begin{eqnarray*}
B(x,y)=\int_{0}^{1}t^{x-1}(1-t)^{y-1}{\rm d}t.
\end{eqnarray*}
or, in terms of gamma function
\begin{eqnarray}
B(x,y)=\frac{\Gamma(x)\Gamma(y)}{\Gamma(x+y)}. \label{prop-beta}
\end{eqnarray}
\end{definition}
\begin{definition}\textnormal{\cite{Samko}}\label{psi-RL}
Let $\alpha>0$, $\Omega=[a,b]$ be a finite or infinite interval, $f$ an integrable function defined on $\Omega$ and 
$\psi\in{C}(\Omega)$ an increasing function such that $\psi'(x)\neq{0}$, for all $x\in\Omega$. The left- and right-sided 
$\psi$-Riemann-Liouville fractional integrals of order $\alpha$ of $f$ on $\Omega$ are defined by
\begin{eqnarray}
\mathds{I}^{\alpha;\psi}_{a+}f(x)=\frac{1}{\Gamma(\alpha)}\int_{a}^{x}\psi'(t)\,(\psi(x)-\psi(t))^{\alpha-1}f(t)\,{\rm d}t \label{integral}
\end{eqnarray}
and 
\begin{eqnarray}
\mathds{I}^{\alpha;\psi}_{b-}f(x)=\frac{1}{\Gamma(\alpha)}\int_{x}^{b}\psi'(t)\,(\psi(t)-\psi(x))^{\alpha-1}f(t)\,{\rm d}t,
\label{integral2}
\end{eqnarray}
respectively. For $\alpha\rightarrow{0}$, we have
$$\mathds{I}^{0;\psi}_{a+}f(x)=\mathds{I}^{0;\psi}_{b-}f(x)=f(x).$$
\end{definition}
\begin{definition}\textnormal{\cite{Almeida}}
Let $\alpha>0$, $n\in\mathbb{N}$, $I$ is the interval $-\infty\leq{a}<b\leq{\infty}$, $f,\psi\in{C^n}(I)$ two 
functions such that $\psi$ is increasing and $\psi'(x)\neq{0}$, for all $x\in I$. The left- and right-sided 
$\psi$-Caputo fractional derivatives of $f$ of order $\alpha$ are given by 
$${^{\rm C}\mathds{D}_{a+}^{\alpha;\psi}}f(x)=\mathds{I}_{a+}^{n-\alpha;\psi}
\left(\frac{1}{\psi'(x)}\frac{\rm d}{{\rm d}x}\right)^n f(x)$$
and
$${^{\rm C}\mathds{D}_{b-}^{\alpha;\psi}}f(x)=\mathds{I}_{b-}^{n-\alpha;\psi}
\left(-\frac{1}{\psi'(x)}\frac{\rm d}{{\rm d}x}\right)^n f(x),$$
respectively, where
$$n=[\alpha]+1 \quad \mbox{for}\quad \alpha\notin\mathbb{N}, \quad\quad 
n=\alpha\quad \mbox{for} \quad \alpha\in\mathbb{N}.$$
To simplify notation, we will use the abbreviated notation
$$f^{[n]}_{\psi}(x)=\left(\frac{1}{\psi'(x)}\frac{\rm d}{{\rm d}x}\right)^n f(x).$$
\end{definition}
\begin{property}\textnormal{\cite{Almeida}}
Let $f\in C^{n}[a,b]$, $\alpha>0$ and $\delta>0,$ 
\begin{enumerate}

\item $f(x)=(\psi(x)-\psi(a))^{\delta-1}$, then

$$\mathds{I}^{\alpha;\psi}_{a+}f(x)=\frac{\Gamma(\delta)}{\Gamma(\alpha+\delta)}(\psi(x)-\psi(a))^{\alpha+\delta-1}.$$


\item $\displaystyle \mathds{I}^{\alpha;\psi}_{a+}{^{C}\mathds{D}^{\alpha;\psi}_{a+}}f(x)=f(x)-
\sum_{k=0}^{n-1}f^{[k]}_{\psi}(a)\,\frac{(\psi(x)-\psi(a))^k}{k!},$ where $n-1<\alpha<n$ with
$n\in\mathbb{N}.$
\end{enumerate}
\end{property}
\section{Main results}
\label{Sec:3}

In this section, we present properties associated with fractional integral operators involving the
three-parameters Mittag-Leffler function in the kernels with respect to another function \cite{Yang}.
These operators were motivated by $\psi$-Riemann-Liouville fractional integrals containing in its 
kernels the three-parameters Mittag-Leffler function.
\begin{definition}\textnormal{\cite{Yang}}
\label{op-ML}
Let $\alpha,\gamma,\rho,\omega\in\mathbb{R}$ with $\alpha>0$ and $\rho>0$ and let $\Omega=[a,b]$ 
be a finite or infinite interval of the real axis $\mathbb{R}$, $f$ an integrable 
function defined on $\Omega$ and $\psi\in C(\Omega)$ an increasing function such that $\psi'(x)\neq{0}$, 
for all $x\in\Omega$. The left- and right-sided fractional integral operators involving the
three-parameters Mittag-Leffler function in the kernels with respect to another function are
defined by
\begin{eqnarray}
\mathds{E}_{\rho,\alpha,\omega;a+}^{\gamma;\psi}f(x)=\int_{a}^{x}\psi'(t)\,(\psi(x)-\psi(t))^{\alpha-1}
E_{\rho,\alpha}^{\gamma}[\omega(\psi(x)-\psi(t))^{\rho}]f(t)\,{\rm d}t
\end{eqnarray}
and
\begin{eqnarray*}
\mathds{E}_{\rho,\alpha,\omega;b-}^{\gamma;\psi}f(x)=\int_{x}^{b}\psi'(t)\,(\psi(t)-\psi(x))^{\alpha-1}
E_{\rho,\alpha}^{\gamma}[\omega(\psi(t)-\psi(x))^{\rho}]f(t)\,{\rm d}t,
\end{eqnarray*}
respectively. In particular, if $\gamma=0$ we have the fractional integrals given by \textnormal{Eq.(\ref{integral})} and
\textnormal{Eq.(\ref{integral2})}, this is,
\begin{eqnarray*}
\mathds{E}_{\rho,\alpha,\omega;a+}^{0;\psi}f(x)=\mathds{I}_{a+}^{\alpha;\psi}f(x) \quad\quad
\textnormal{and} \quad\quad \mathds{E}_{\rho,\alpha,\omega;b-}^{0;\psi}f(x)=\mathds{I}_{b-}^{\alpha;\psi}f(x) .
\end{eqnarray*}
If $\alpha\rightarrow{0}$ and $\gamma=0$, we have
\begin{eqnarray*}
\mathds{E}_{\rho,0,\omega;a+}^{0;\psi}f(x)=f(x) \quad\quad
\textnormal{and} \quad\quad \mathds{E}_{\rho,0,\omega;b-}^{0;\psi}f(x)=f(x) .
\end{eqnarray*}
\end{definition}
We prove some properties of the left-sided fractional operator $\mathds{E}_{\rho,\alpha,\omega;a+}^{\gamma;\psi}$. 
The corresponding results for $\mathds{E}_{\rho,\alpha,\omega;b-}^{\gamma;\psi}$ can be derived analogously.
The first result yields the linearity of the integral operators $\mathds{E}_{\rho,\alpha,\omega;a+}^{\gamma;\psi}$
and $\mathds{E}_{\rho,\alpha,\omega;b-}^{\gamma;\psi}$.
\begin{theorem}
\label{th-1}
Let $\alpha,\gamma,\rho,\omega\in\mathbb{R}$, with $\alpha>0$ and $\rho>0$. Also let $f$ and $g$ two functions
and $\lambda,\mu$ are arbitrary real constants, then
\begin{eqnarray*}
\mathds{E}_{\rho,\alpha,\omega;a+}^{\gamma;\psi}[\lambda f(x)\pm\mu g(x)]=
\lambda\mathds{E}_{\rho,\alpha,\omega;a+}^{\gamma;\psi}f(x)\pm\mu\mathds{E}_{\rho,\alpha,\omega;a+}^{\gamma;\psi}g(x)
\end{eqnarray*}
and
\begin{eqnarray*}
\mathds{E}_{\rho,\alpha,\omega;b-}^{\gamma;\psi}[\lambda f(x)\pm\mu g(x)]=
\lambda\mathds{E}_{\rho,\alpha,\omega;b-}^{\gamma;\psi}f(x)\pm\mu\mathds{E}_{\rho,\alpha,\omega;b-}^{\gamma;\psi}g(x).
\end{eqnarray*}
\end{theorem}
\begin{proof}
The result follows from the fact these integral operators are linear.
\end{proof}
The second result consists in calculating the fractional integrals $\mathds{E}_{\rho,\alpha,\omega;a+}^{\gamma;\psi}$
and $\mathds{E}_{\rho,\alpha,\omega;b-}^{\gamma;\psi}$ of a power function.
\begin{lemma}\label{lemma1}
Let $\alpha,\beta,\gamma,\rho,\omega\in\mathbb{R}$, with $\alpha>0$, $\beta>0$ and $\rho>0$. Then,
\begin{eqnarray}
\mathds{E}_{\rho,\alpha,\omega;a+}^{\gamma;\psi}[(\psi(x)-\psi(a))^{\beta-1}]=\Gamma(\beta)(\psi(x)-\psi(a))^{\alpha+\beta-1}
E_{\rho,\alpha+\beta}^{\gamma}[\omega(\psi(x)-\psi(a))^{\rho}] \label{lemma}
\end{eqnarray}
and
\begin{eqnarray*}
\mathds{E}_{\rho,\alpha,\omega;b-}^{\gamma;\psi}[(\psi(b)-\psi(x))^{\beta-1}]=\Gamma(\beta)(\psi(b)-\psi(x))^{\alpha+\beta-1}
E_{\rho,\alpha+\beta}^{\gamma}[\omega(\psi(b)-\psi(x))^{\rho}].
\end{eqnarray*}
\end{lemma}
\begin{proof}
From \textbf{Definition \ref{op-ML}}, we can write
\begin{eqnarray*}
\mathds{E}_{\rho,\alpha,\omega;a+}^{\gamma;\psi}[(\psi(x)-\psi(a))^{\beta-1}]=
\int_{a}^{x}\psi'(t)\sum_{k=0}^{\infty}\left[\frac{(\gamma)_k}{\Gamma(\rho{k}+\alpha)}\frac{\omega^k(\psi(x)-\psi(t))^{\rho{k}+\alpha-1}}{k!}\right](\psi(t)-\psi(a))^{\beta-1}\,{\rm d}t.
\end{eqnarray*}
Making the change of variable $\displaystyle \tau=\frac{\psi(t)-\psi(a)}{\psi(x)-\psi(a)}$ and 
since, the entire Mittag-Leffler function is uniformly convergent, we have interchange the order 
of integration and summation, to get
\begin{eqnarray*}
\mathds{E}_{\rho,\alpha,\omega;a+}^{\gamma;\psi}[(\psi(x)-\psi(a))^{\beta-1}]=
\sum_{k=0}^{\infty}\frac{\omega^k(\gamma)_k}{\Gamma(\rho{k}+\alpha)k!}(\psi(x)-\psi(a))^{\rho{k}+\alpha+\beta-1}
\int_{0}^{1}(1-\tau)^{\rho{k}+\alpha-1}\tau^{\beta-1}{\rm d}\tau,
\end{eqnarray*}
which, in accordance with Eq.(\ref{prop-beta}), yields Eq.(\ref{lemma}).
\end{proof}
The following assertion, which yields the boundedness of the fractional integration operator
$\mathds{E}_{\rho,\alpha,\omega;a+}^{\gamma;\psi}$ from the space $C_{\nu,\psi}$.
\begin{theorem}
\label{th-2}
Let $\alpha,\gamma,\rho,\omega\in\mathbb{R}$ with $\alpha>0$ and $b>a$. The operators
$\mathds{E}_{\rho,\alpha,\omega;a+}^{\gamma;\psi}$ and $\mathds{E}_{\rho,\alpha,\omega;b-}^{\gamma;\psi}$ 
are bounded on $C_{\nu,\psi}$
\begin{eqnarray*}
{\lVert\mathds{E}_{\rho,\alpha,\omega;a+}^{\gamma;\psi}f\rVert}_{C_{\nu,\psi}[a,b]}\leq
M{\lVert f\rVert}_{C_{\nu,\psi}[a,b]}
\quad\quad \mbox{and} \quad\quad
{\lVert\mathds{E}_{\rho,\alpha,\omega;b-}^{\gamma;\psi}f\rVert}_{C_{\nu,\psi}[a,b]}\leq
M{\lVert f\rVert}_{C_{\nu,\psi}[a,b]},
\end{eqnarray*}
where
\begin{eqnarray}
M=|(\psi(b)-\psi(a))^{\alpha}\,{E}_{\rho,\alpha+1}^{\gamma}[\omega\,(\psi(b)-\psi(a))^{\rho}]|. \label{series}
\end{eqnarray}
\end{theorem}
\begin{proof}
According to \textnormal{Eq.(\ref{norm})}, we have
\begin{eqnarray*}
{\lVert\mathds{E}_{\rho,\alpha,\omega;a+}^{\gamma;\psi}f\rVert}_{C_{\nu,\psi}[a,b]}&=&
{\lVert(\psi(x)-\psi(a))^{\nu}\,\mathds{E}_{\rho,\alpha,\omega;a+}^{\gamma;\psi}f\rVert}_{C[a,b]}\\
&=&\max_{x\in[a,b]}|(\psi(x)-\psi(a))^{\nu}\,\mathds{E}_{\rho,\alpha,\omega;a+}^{\gamma;\psi}f|\\
&\leq &{\lVert f\rVert}_{C_{\nu,\psi}[a,b]}\,
\max_{x\in[a,b]}\int_{a}^{x}|\psi'(t)\,(\psi(x)-\psi(t))^{\alpha-1}E_{\rho,\alpha}^{\gamma}[\omega(\psi(x)-\psi(t))^{\rho}]|\,
{\rm d}t\\
&=&{\lVert f\rVert}_{C_{\nu,\psi}[a,b]}\,\max_{x\in[a,b]}|(\psi(x)-\psi(a))^{\alpha}E_{\rho,\alpha+1}^{\gamma}
[\omega(\psi(x)-\psi(a))^{\rho}]|\\
&\leq &|(\psi(b)-\psi(a))^{\alpha}\,E_{\rho,\alpha+1}^{\gamma}[\omega(\psi(b)-\psi(a))^{\rho}]|\,
{\lVert f\rVert}_{C_{\nu,\psi}[a,b]}\\
&=&M{\lVert f\rVert}_{C_{\nu,\psi}[a,b]},
\end{eqnarray*}
where $M$ is given by Eq.(\ref{series}). Consider the following relation with $x\in\mathbb{R}$, $x>0$ and
$(x+b)\in\mathbb{R}\backslash \{0,-1,-2,\dots\}$, \cite{Mubeen}
$$\lim_{x \to \infty}\frac{\Gamma(x+a)}{\Gamma(x+b)}x^{b-a}=1.$$
We denote by $c_k$ the $k$th term of the series in Eq.(\ref{series}), then by ratio test the series converges,
\begin{eqnarray*}
\lim_{k \to \infty}\biggl|\frac{c_{k+1}}{c_k}\biggl| &=&\lim_{k \to \infty}\biggl|\frac{\omega(\gamma+k)\,
\Gamma(\rho{k}+\alpha+1)}{(k+1)\,\Gamma(\rho k+\rho+\alpha+1)}\biggl|(\psi(b)-\psi(a))^{\rho}\\
&=&\lim_{k \to \infty}\left[\frac{(\gamma+k)}{(k+1)\,(\rho k)^{\rho}}\right]|\omega|(\psi(b)-\psi(a))^{\rho}\rightarrow{0}.
\end{eqnarray*}
\end{proof}
The next result shows the composition of the fractional integral of a function with respect to
another function $\mathds{I}_{a+}^{\beta;\psi}$ and fractional integral operator involving the
three-parameters Mittag-Leffler function in the kernel with respect to another function
$\mathds{E}_{\rho,\alpha,\omega;a+}^{\gamma;\psi}$.
\begin{theorem}
\label{th-3}
Let $\alpha,\beta,\gamma,\rho,\omega\in\mathbb{R}$ with $\alpha>0$, $\beta>0$ and $\rho>0$. Then,
\begin{eqnarray}
\mathds{E}_{\rho,\alpha,\omega;a+}^{\gamma;\psi}\mathds{I}_{a+}^{\beta;\psi}f(x)=
\mathds{E}_{\rho,\alpha+\beta,\omega;a+}^{\gamma;\psi}f(x)=
\mathds{I}_{a+}^{\beta;\psi}\mathds{E}_{\rho,\alpha,\omega;a+}^{\gamma;\psi}f(x),
\end{eqnarray}
and
\begin{eqnarray*}
\mathds{E}_{\rho,\alpha,\omega;b-}^{\gamma;\psi}\mathds{I}_{b-}^{\beta;\psi}f(x)=
\mathds{E}_{\rho,\alpha+\beta,\omega;b-}^{\gamma;\psi}f(x)=
\mathds{I}_{b-}^{\beta;\psi}\mathds{E}_{\rho,\alpha,\omega;b-}^{\gamma;\psi}f(x).
\end{eqnarray*}
\end{theorem}
\begin{proof}
From \textbf{Definition \ref{op-ML}}, we have
\begin{eqnarray*}
\mathds{E}_{\rho,\alpha,\omega;a+}^{\gamma;\psi}\mathds{I}_{a+}^{\beta;\psi}f(x)&=&
\int_{a}^{x}\psi'(t)\,(\psi(x)-\psi(t))^{\alpha-1}E_{\rho,\alpha}^{\gamma}[\omega(\psi(x)-\psi(t))^{\rho}]\\
&\times &\left[\frac{1}{\Gamma(\beta)}\int_{a}^{t}\psi'(u)\,(\psi(t)-\psi(u))^{\beta-1}f(u)\,{\rm d}u\right]{\rm d}t.
\end{eqnarray*}
Applying the Dirichlet formula to interchange the order of integration, we obtain
\begin{eqnarray*}
\mathds{E}_{\rho,\alpha,\omega;a+}^{\gamma;\psi}\mathds{I}_{a+}^{\beta;\psi}f(x)&=&
\frac{1}{\Gamma(\beta)}\int_{a}^{x}\psi'(u)f(u)\,{\rm d}u\int_{u}^{x}\psi'(t)\,(\psi(x)-\psi(t))^{\alpha-1}\\
&\times & E_{\rho,\alpha}^{\gamma}[\omega(\psi(x)-\psi(t))^{\rho}](\psi(t)-\psi(u))^{\beta-1}{\rm d}t,
\end{eqnarray*}
and by changing the variable $\displaystyle \tau=\frac{\psi(t)-\psi(u)}{\psi(x)-\psi(u)}$, 
in the above second integral and rearranging, we find that
\begin{eqnarray*}
\mathds{E}_{\rho,\alpha,\omega;a+}^{\gamma;\psi}\mathds{I}_{a+}^{\beta;\psi}f(x)=
\frac{1}{\Gamma(\beta)}\int_{a}^{x}\psi'(u)f(u)\,{\rm d}u\sum_{k=0}^{\infty}(\psi(x)-\psi(u))^{\rho{k}+\alpha+\beta-1}
\frac{\omega^{k}(\gamma)_k}{\Gamma(\rho{k}+\alpha)k!}\int_{0}^{1}(1-\tau)^{\rho{k}+\alpha-1}\tau^{\beta-1}{\rm d}\tau.
\end{eqnarray*}
By Eq.(\ref{prop-beta}), we can write
\begin{eqnarray*}
\mathds{E}_{\rho,\alpha,\omega;a+}^{\gamma;\psi}\mathds{I}_{a+}^{\beta;\psi}f(x)&=&
\int_{a}^{x}\psi'(u)(\psi(x)-\psi(u))^{\alpha+\beta-1}E_{\rho,\alpha+\beta}^{\gamma}[\omega(\psi(x)-\psi(u))^{\rho}]
f(u)\,{\rm d}u\\
&=&\mathds{E}_{\rho,\alpha+\beta,\omega;a+}^{\gamma;\psi}f(x).
\end{eqnarray*}
The proof of $\mathds{I}_{a+}^{\beta;\psi}\mathds{E}_{\rho,\alpha,\omega;a+}^{\gamma;\psi}f(x)=
\mathds{E}_{\rho,\alpha+\beta,\omega;a+}^{\gamma;\psi}f(x)$ is similar.
\end{proof}
The following assertion for fractional integral operator involving the
three-parameters Mittag-Leffler function in the kernel with respect to another function 
is the validity of the semigroup property.
\begin{theorem}
\label{th-4}
Let $\alpha,\beta,\gamma,\nu,\rho,\sigma,\omega\in\mathbb{R}$ with $\alpha>0$, $\nu>0$ and $\rho>0$. Then,
\begin{eqnarray}
\mathds{E}_{\rho,\alpha,\omega;a+}^{\gamma;\psi}\mathds{E}_{\rho,\nu,\omega;a+}^{\sigma;\psi}f(x)=
\mathds{E}_{\rho,\alpha+\nu,\omega;a+}^{\gamma+\sigma;\psi}f(x)=
\mathds{E}_{\rho,\nu,\omega;a+}^{\sigma;\psi}\mathds{E}_{\rho,\alpha,\omega;a+}^{\gamma;\psi}f(x)
\end{eqnarray}
and
\begin{eqnarray*}
\mathds{E}_{\rho,\alpha,\omega;b-}^{\gamma;\psi}\mathds{E}_{\rho,\nu,\omega;b-}^{\sigma;\psi}f(x)=
\mathds{E}_{\rho,\alpha+\nu,\omega;b-}^{\gamma+\sigma;\psi}f(x)=
\mathds{E}_{\rho,\nu,\omega;b-}^{\sigma;\psi}\mathds{E}_{\rho,\alpha,\omega;b-}^{\gamma;\psi}f(x).
\end{eqnarray*}
\end{theorem}
\begin{proof}
Considering the \textbf{Definition \ref{op-ML}}, we have
\begin{eqnarray*}
\mathds{E}_{\rho,\alpha,\omega;a+}^{\gamma;\psi}\mathds{E}_{\rho,\nu,\omega;a+}^{\sigma;\psi}f(x)&=&
\int_{a}^{x}\psi'(t)(\psi(x)-\psi(t))^{\alpha-1}E_{\rho,\alpha}^{\gamma}[\omega(\psi(x)-\psi(t))^{\rho}]\\
&\times &
\left[\int_{a}^{t}\psi'(u)(\psi(t)-\psi(u))^{\nu-1}E_{\rho,\nu}^{\sigma}[\omega(\psi(t)-\psi(u))^{\rho}]f(u)\,{\rm d}u\right]
{\rm d}t.
\end{eqnarray*}
Interchanging the order of integration, we can write
\begin{eqnarray*}
\mathds{E}_{\rho,\alpha,\omega;a+}^{\gamma;\psi}\mathds{E}_{\rho,\nu,\omega;a+}^{\sigma;\psi}f(x)&=&
\int_{a}^{x}\psi'(u)f(u)\,{\rm d}u\int_{u}^{x}\psi'(t)(\psi(x)-\psi(t))^{\alpha-1}
E_{\rho,\alpha}^{\gamma}[\omega(\psi(x)-\psi(t))^{\rho}]\\
&\times &(\psi(t)-\psi(u))^{\nu-1}
E_{\rho,\nu}^{\sigma}[\omega(\psi(t)-\psi(u))^{\rho}]{\rm d}t.
\end{eqnarray*}
Taking the same variable change as \textbf{Theorem \ref{th-3}}, we obtain
\begin{eqnarray*}
\mathds{E}_{\rho,\alpha,\omega;a+}^{\gamma;\psi}\mathds{E}_{\rho,\nu,\omega;a+}^{\sigma;\psi}f(x)&=&
\int_{a}^{x}\psi'(u)f(u)(\psi(x)-\psi(u))^{\alpha+\nu-1}E_{\rho,\alpha}^{\gamma}[\omega(\psi(x)-\psi(u))^{\rho}]
E_{\rho,\nu}^{\sigma}[\omega(\psi(x)-\psi(u))^{\rho}]\\
&\times &\left[\int_{0}^{1}(1-\tau)^{\rho{k}+\alpha-1}\tau^{\rho{m}+\nu-1}{\rm d}
\tau\right]{\rm d}u.
\end{eqnarray*}
From \textbf{Definition \ref{beta}}, we find
\begin{eqnarray*}
\mathds{E}_{\rho,\alpha,\omega;a+}^{\gamma;\psi}\mathds{E}_{\rho,\nu,\omega;a+}^{\sigma;\psi}f(x)=
\int_{a}^{x}\psi'(u)f(u)\,(\psi(x)-\psi(u))^{\alpha+\nu-1}
\sum_{k=0}^{\infty}\sum_{m=0}^{\infty}\frac{\omega^{k+m}(\gamma)_k(\sigma)_m\,(\psi(x)-\psi(u))^{\rho(k+m)}}
{k!\,m!\,{\Gamma(\rho(k+m)+\alpha+\nu)}}\,{\rm d}u.
\end{eqnarray*}
Let $k\rightarrow{k-m}$, then
\begin{eqnarray*}
\mathds{E}_{\rho,\alpha,\omega;a+}^{\gamma;\psi}\mathds{E}_{\rho,\nu,\omega;a+}^{\sigma;\psi}f(x)&=&
\int_{a}^{x}\psi'(u)f(u)\,(\psi(x)-\psi(u))^{\alpha+\nu-1}\sum_{k=m}^{\infty}\sum_{m=0}^{\infty}
\frac{\omega^{k}(\gamma)_{k-m}(\sigma)_m\,(\psi(x)-\psi(u))^{\rho{k}}}{(k-m)!\,m!\,\Gamma(\rho{k}+\alpha+\nu)}\,{\rm d}u\\
&=&\int_{a}^{x}\psi'(u)f(u)\,(\psi(x)-\psi(u))^{\alpha+\nu-1}\sum_{k=0}^{\infty}\sum_{m=0}^{k}{{k}\choose{m}}
\frac{\omega^{k}(\gamma)_{k-m}(\sigma)_m\,(\psi(x)-\psi(u))^{\rho{k}}}{\Gamma(\rho{k}+\alpha+\nu)k!}\,{\rm d}u.
\end{eqnarray*}
Using the relation
\begin{eqnarray*}
(\gamma+\sigma)_{k}=\sum_{m=0}^{k}{{k}\choose{m}}(\gamma)_{k-m}(\sigma)^{m},
\end{eqnarray*}
we have
\begin{eqnarray*}
\mathds{E}_{\rho,\alpha,\omega;a+}^{\gamma;\psi}\mathds{E}_{\rho,\nu,\omega;a+}^{\sigma;\psi}f(x)&=&
\int_{a}^{x}\psi'(u)f(u)\,(\psi(x)-\psi(u))^{\alpha+\nu-1}\sum_{k=0}^{\infty}
\frac{(\gamma+\sigma)_{k}}{\Gamma(\rho{k}+\alpha+\nu)}\frac{[\omega(\psi(x)-\psi(u))^{\rho}]^{k}}{k!}\,{\rm d}u\\
&=&\int_{a}^{x}\psi'(u)f(u)\,(\psi(x)-\psi(u))^{\alpha+\nu-1}
E_{\rho,\alpha+\nu}^{\gamma+\sigma}[\omega(\psi(x)-\psi(u))^{\rho}]\,{\rm d}u\\
&=&\mathds{E}_{\rho,\alpha+\nu,\omega;a+}^{\gamma+\sigma;\psi}f(x).
\end{eqnarray*}
The proof of $\mathds{E}_{\rho,\nu,\omega;a+}^{\sigma;\psi}\mathds{E}_{\rho,\alpha,\omega;a+}^{\gamma;\psi}f(x)=
\mathds{E}_{\rho,\alpha+\nu,\omega;a+}^{\gamma+\sigma;\psi}f(x)$ goes along similar lines.
\end{proof}
\section{Cauchy problem}
\label{Sec:4}

In this section, we convert the initial value problem for the differential equation (Cauchy problem)
into an equivalent Volterra integral equation. We obtain the solution of the Cauchy 
problem using the successive approximations.
\begin{theorem}\textnormal{\cite{Almeida2}}
\label{th-5}
Consider the initial value problem:
\begin{eqnarray}
\begin{dcases}
{^{\rm{C}}\mathds{D}_{a+}^{\beta,\psi}}u(x)=f(x,u(x)), \quad\quad x\in[a,b], \label{C1}\\
\left(\frac{1}{\psi'(x)}\frac{{\rm d}}{{\rm d}x}\right)^{i}u(x)\biggl|_{x=a}=b_i, \quad i=0,1,\cdots,n-1,
\end{dcases}
\end{eqnarray}
where
\begin{enumerate}
\item $0<\beta\notin\mathbb{N}$ and $n=[\beta]+1$,
\item $b_i$, $i=0,1,\ldots,n-1$, are fixed reals,
\item $u\in{C^{n-1}[a,b]}$ such that ${^{\rm{C}}\mathds{D}_{a+}^{\beta,\psi}}u$ exists and is
continuous in $[a,b]$,
\item $f : [a,b]\times\mathbb{R}\rightarrow\mathbb{R}$ is continuous.
\end{enumerate}
The Cauchy problem \textnormal{(\ref{C1})} is equivalent to the following Volterra integral
equation
\begin{eqnarray}
u(x)=\sum_{i=0}^{n-1}{b_i}\frac{(\psi(x)-\psi(a))^{i}}{i!}+\mathds{I}_{a+}^{\beta;\psi}f(x,u(x)). \label{sol-1}
\end{eqnarray}
If $f$ is Lipschitz continuous with respect to the second variable, then exists a unique solution to problem 
\textnormal{(\ref{C1})} on interval $[a,a+h]\subseteq [a,b]$.
\end{theorem}
From \textbf{Theorem \ref{th-5}}, we have the following lemma and theorem as particular cases.
\begin{lemma}
\label{lemma-2}
Let $\alpha,\beta,\gamma,\rho,\lambda,\omega\in\mathbb{R}$ with $\alpha>0$, $\beta>0$ and $\rho>0$.
We consider the Cauchy problem with initial conditions:
\begin{eqnarray}
\begin{dcases}
{^{\rm{C}}\mathds{D}_{a+}^{\beta,\psi}}u(x)=\lambda\mathds{E}_{\rho,\alpha,\omega;a+}^{\gamma;\psi}u(x)+f(x),\label{Cauchy-1}\\
\left(\frac{1}{\psi'(x)}\frac{{\rm d}}{{\rm d}x}\right)^{i}u(x)\biggl|_{x=a}=b_i, \quad (b_i\in\mathbb{R};\,
i=0,1,\cdots,n-1). \\
\end{dcases}
\end{eqnarray}
We suppose that $f\in{C[a,b]},\,\,a\leq{x}\leq{b}$, then by \textnormal{\textbf{Theorem \ref{th-5}}}, 
the Cauchy problem \textnormal{(\ref{Cauchy-1})} is equivalent in the space $C^{n-1}[a,b]$
to the Volterra integral equation of the second kind
\begin{eqnarray}
u(x)&=&\sum_{i=0}^{\infty}{b_i}\frac{(\psi(x)-\psi(a))^{i}}{i!}+
\lambda\int_{a}^{x}\psi'(t)(\psi(x)-\psi(t))^{\alpha+\beta-1}E_{\rho,\alpha+\beta}^{\gamma}[\omega(\psi(x)-\psi(t))^{\rho}]
u(t){\rm d}t \nonumber\\
&+&\frac{1}{\Gamma(\beta)}\int_{a}^{x}\psi'(t)(\psi(x)-\psi(t))^{\beta-1}f(t){\rm d}t. \label{sol-Cauchy}
\end{eqnarray}
\end{lemma}
\begin{proof}
From \textnormal{Eq.(\ref{sol-1})} with $f(x,u(x))=\lambda\mathds{E}_{\rho,\alpha,\omega;a+}^{\gamma;\psi}u(x)+f(x)$, 
we have
\begin{eqnarray*}
u(x)=\sum_{i=0}^{n-1}{b_i}\frac{(\psi(x)-\psi(a))^{i}}{i!}+
\lambda\mathds{I}^{\beta;\psi}_{a+}\mathds{E}_{\rho,\alpha,\omega;a+}^{\gamma;\psi}u(x)+
\mathds{I}^{\beta;\psi}_{a+}f(x).
\end{eqnarray*}
According to \textbf{Definition \ref{psi-RL}} and \textbf{Theorem \ref{th-3}}, we obtain
\textnormal{Eq.(\ref{sol-Cauchy})}.
\end{proof}
\begin{theorem}
\label{th-6}
Let $\alpha,\beta,\gamma,\rho,\lambda,\omega\in\mathbb{R}$ with $\alpha>0$, $\beta>0$ and $\rho>0$.
The solution of \textnormal{Eq.(\ref{sol-Cauchy})} is given by
\begin{eqnarray}
u(x)&=&\sum_{i=0}^{n-1}{b_i}(\psi(x)-\psi(a))^{i}\sum_{j=0}^{\infty}\lambda^{j}(\psi(x)-\psi(a))^{j(\alpha+\beta)}
E_{\rho,j(\alpha+\beta)+i+1}^{j{\gamma}}[\omega(\psi(x)-\psi(a))^{\rho}]\nonumber\\
&+&\sum_{j=0}^{\infty}{\lambda^j}\mathds{E}_{\rho,j(\alpha+\beta)+\beta,\omega;a+}^{j\gamma;\psi}f(x). \label{sol}
\end{eqnarray}
\end{theorem}
\begin{proof}
According to successive approximations method, we set
\begin{eqnarray}
u_0(x)=\sum_{i=0}^{\infty}{b_i}\frac{(\psi(x)-\psi(a))^{i}}{i!} \label{u_0}
\end{eqnarray}
and
\begin{eqnarray}
u_m(x)=u_0(x)+\lambda\mathds{I}^{\beta;\psi}_{a+}\mathds{E}_{\rho,\alpha,\omega;a+}^{\gamma;\psi}u_{m-1}(x)+
\mathds{I}^{\beta;\psi}_{a+}f(x),\quad\quad m\in\mathbb{N}. \label{u_m}
\end{eqnarray}
Using Eq.(\ref{u_m}) with $m=1,2,\ldots$, we find
\begin{eqnarray*}
u_1(x)&=&u_0(x)+\lambda\mathds{I}^{\beta;\psi}_{a+}\mathds{E}_{\rho,\alpha,\omega;a+}^{\gamma;\psi}u_{0}(x)+
\mathds{I}^{\beta;\psi}_{a+}f(x)\\
u_2(x)&=&u_0(x)+\lambda\mathds{I}^{\beta;\psi}_{a+}\mathds{E}_{\rho,\alpha,\omega;a+}^{\gamma;\psi}u_{1}(x)+
\mathds{I}^{\beta;\psi}_{a+}f(x)\\
&=&u_0(x)+\lambda\mathds{I}^{\beta;\psi}_{a+}\mathds{E}_{\rho,\alpha,\omega;a+}^{\gamma;\psi}[u_0(x)+
\lambda\mathds{I}^{\beta;\psi}_{a+}\mathds{E}_{\rho,\alpha,\omega;a+}^{\gamma;\psi}u_{0}(x)+\mathds{I}^{\beta;\psi}_{a+}f(x)]
+\mathds{I}^{\beta;\psi}_{a+}f(x)\\
&=&u_0(x)+\lambda\mathds{E}_{\rho,\alpha+\beta,\omega;a+}^{\gamma;\psi}u_0(x)+
{\lambda^2}\mathds{E}_{\rho,2(\alpha+\beta),\omega;a+}^{2\gamma;\psi}u_0(x)+
\lambda\mathds{E}_{\rho,(\alpha+\beta)+\beta,\omega;a+}^{\gamma;\psi}f(x)+
\mathds{I}_{a+}^{\beta;\psi}f(x)\\
u_3(x)&=&u_0(x)+\lambda\mathds{I}^{\beta;\psi}_{a+}\mathds{E}_{\rho,\alpha,\omega;a+}^{\gamma;\psi}u_{2}(x)+
\mathds{I}^{\beta;\psi}_{a+}f(x)\\
&=&u_0(x)+\lambda\mathds{I}^{\beta;\psi}_{a+}\mathds{E}_{\rho,\alpha,\omega;a+}^{\gamma;\psi}[u_0(x)+
\lambda\mathds{E}_{\rho,\alpha+\beta,\omega;a+}^{\gamma;\psi}u_0(x)+
{\lambda^2}\mathds{E}_{\rho,2(\alpha+\beta),\omega;a+}^{2\gamma;\psi}u_0(x)\\
&+&
\lambda\mathds{E}_{\rho,(\alpha+\beta)+\beta,\omega;a+}^{\gamma;\psi}f(x)+
\mathds{I}_{a+}^{\beta;\psi}f(x)]+\mathds{I}^{\beta;\psi}_{a+}f(x)\\
&=&u_0(x)+\lambda\mathds{E}_{\rho,\alpha+\beta,\omega;a+}^{\gamma;\psi}u_0(x)+
{\lambda^2}\mathds{E}_{\rho,2(\alpha+\beta),\omega;a+}^{2\gamma;\psi}u_0(x)+
{\lambda^3}\mathds{E}_{\rho,3(\alpha+\beta),\omega;a+}^{3\gamma;\psi}u_0(x)\\
&+&
\lambda\mathds{E}_{\rho,(\alpha+\beta)+\beta,\omega;a+}^{\gamma;\psi}f(x)
+{\lambda^2}\mathds{E}_{\rho,2(\alpha+\beta)+\beta;\omega;a+}^{\gamma;\psi}f(x)+
\mathds{I}_{a+}^{\beta;\psi}f(x)\\
&\vdots&\\
u_m(x)&=&u_0(x)+\sum_{j=1}^{m}{\lambda^j}\mathds{E}_{\rho,j(\alpha+\beta),\omega;a+}^{j\gamma,\psi}u_0(x)
+\sum_{j=1}^{m-1}{\lambda^j}\mathds{E}_{\rho,j(\alpha+\beta)+\beta,\omega;a+}^{j\gamma;\psi}f(x)+
\mathds{I}_{a+}^{\beta;\psi}f(x).
\end{eqnarray*}
From Eq.(\ref{u_0}), \textbf{Lemma \ref{lemma1}} and the particular case of \textbf{Definition \ref{op-ML}},
we have
\begin{eqnarray*}
u_m(x)=\sum_{i=0}^{n-1}{b_i}(\psi(x)-\psi(a))^{i}\sum_{j=0}^{m}{\lambda^j}
{E}_{\rho,j(\alpha+\beta)+i+1}^{j\gamma}[\omega(\psi(x)-\psi(a))^{\rho}]+
\sum_{j=0}^{m-1}{\lambda^j}\mathds{E}_{\rho,j(\alpha+\beta)+\beta,\omega;a+}^{j\gamma;\psi}f(x).
\end{eqnarray*}
\end{proof}
Taking $m\rightarrow\infty$ follows Eq.(\ref{sol}).
\subsection{Particular case}
\label{sub:4}

To conclude this section, we consider the following particular case of the Cauchy problem, Eq.(\ref{Cauchy-1}),
taking $f(x)=\xi(\psi(x)-\psi(a))^{\mu-1}E_{\rho,\mu}^{\sigma}[\omega(\psi(x)-\psi(a))^{\rho}]$.
\begin{theorem}
\label{th-7}
Let $\alpha,\beta,\gamma,\mu,\rho,\sigma,\xi,\lambda,\omega\in\mathbb{R}$ with $\alpha>0$, $\beta>0$, $\rho>0$ and
$\mu>0$. Thus, the Cauchy problem
\begin{eqnarray}
\begin{dcases}
{^{\rm{C}}\mathds{D}_{a+}^{\beta,\psi}}u(x)=\lambda\mathds{E}_{\rho,\alpha,\omega;a+}^{\gamma;\psi}u(x)+
\xi\,(\psi(x)-\psi(a))^{\mu-1}E_{\rho,\mu}^{\sigma}[\omega(\psi(x)-\psi(a))^{\rho}],\label{Cauchy-2}\\
\left(\frac{1}{\psi'(x)}\frac{{\rm d}}{{\rm d}x}\right)^{i}u(x)\biggl|_{x=a}=b_i, \quad (b_i\in\mathbb{R};\,
i=0,1,\cdots,n-1) \\
\end{dcases}
\end{eqnarray}
admits a unique solution $u(x)\in{C}^{n-1}[a,b]$, given by
\begin{eqnarray}
u(x)&=&\sum_{i=0}^{n-1}{b_i}(\psi(x)-\psi(a))^{i}\sum_{j=0}^{\infty}\lambda^{j}(\psi(x)-\psi(a))^{j(\alpha+\beta)}
E_{\rho,j(\alpha+\beta)+i+1}^{j{\gamma}}[\omega(\psi(x)-\psi(a))^{\rho}]\label{sol-pc}\\
&+&\xi(\psi(x)-\psi(a))^{\beta+\mu-1}\sum_{j=0}^{\infty}{\lambda^j}(\psi(x)-\psi(a))^{j(\alpha+\beta)}
{E}_{\rho,j(\alpha+\beta)+\beta+\mu}^{j\gamma+\sigma}[\omega(\psi(x)-\psi(a))^{\rho}]. \nonumber
\end{eqnarray}
\end{theorem}
\begin{proof}
Using \textnormal{\textbf{Lemma \ref{lemma-2}}} with $f(x)=\xi(\psi(x)-\psi(a))^{\mu-1}
E_{\rho,\mu}^{\sigma}[\omega(\psi(x)-\psi(a))^{\rho}]$ and the linearity property,\textnormal{\textbf{Theorem \ref{th-1}}}, 
we have
\begin{eqnarray*}
u(x)&=&\sum_{i=0}^{n-1}{b_i}(\psi(x)-\psi(a))^{i}\sum_{j=0}^{\infty}\lambda^{j}(\psi(x)-\psi(a))^{j(\alpha+\beta)}
E_{\rho,j(\alpha+\beta)+i+1}^{j{\gamma}}[\omega(\psi(x)-\psi(a))^{\rho}]\\
&+&\xi\sum_{j=0}^{\infty}\underbrace{\mathds{E}_{j(\alpha+\beta)+\beta,\omega;a+}^{j\gamma;\psi}
\left\{(\psi(x)-\psi(a))^{\mu-1}E_{\rho,\mu}^{\sigma}[\omega(\psi(x)-\psi(a))^{\rho}]\right\}}_{(\star)}.
\end{eqnarray*}
Eq.($\star$) can be proved directly by using the proof of \textnormal{\textbf{Theorem \ref{th-3}}}, which
yields Eq.(\ref{sol-pc}). If $\psi(x)=x$, we recover the result presented in \cite{Kilbas2002}.
\end{proof}
\section{The inverse operator}
\label{Sec:5}

In this section, we construct the left inverse operator $\mathbf{D}_{a+}^{\gamma;\psi}$ of 
the operator $\mathds{E}_{\rho,\mu,\omega;a+}^{\gamma;\psi}$.

\begin{definition}

Let $\alpha,\beta,\gamma,\mu,\rho,\omega\in\mathbb{R}$ with $\alpha>0$ and $\rho>0$. We define the left
inverse operators $\mathbf{D}_{a+}^{\gamma;\psi}$ and $\mathbf{D}_{b-}^{\gamma;\psi}$ of the operators
$\mathds{E}_{\rho,\mu,\omega;a+}^{\gamma;\psi}$ and $\mathds{E}_{\rho,\mu,\omega;b-}^{\gamma;\psi}$, respectively, 
as follows:
\begin{eqnarray*}
\mathbf{D}_{a+}^{\gamma;\psi}f(x)={^{\rm{C}}{\mathds{D}_{a+}^{\alpha;\psi}}}
\mathds{E}_{\rho,\alpha-\mu,\omega;a+}^{-\gamma;\psi}f(x)
\end{eqnarray*}
and
\begin{eqnarray*}
\mathbf{D}_{b-}^{\gamma;\psi}f(x)={^{\rm{C}}{\mathds{D}_{b-}^{\alpha;\psi}}}
\mathds{E}_{\rho,\alpha-\mu,\omega;b-}^{-\gamma;\psi}f(x).
\end{eqnarray*}
\end{definition}
In fact, we have
\begin{eqnarray*}
\mathbf{D}_{a+}^{\gamma;\psi}\mathds{E}_{\rho,\mu,\omega;a+}^{\gamma;\psi}f(x)=
{^{\rm{C}}\mathds{D}_{a+}^{\alpha;\psi}}\mathds{E}_{\rho,\alpha-\mu,\omega;a+}^{-\gamma;\psi}
\mathds{E}_{\rho,\mu,\omega;a+}^{\gamma;\psi}f(x)=
{^{\rm{C}}\mathds{D}_{a+}^{\alpha;\psi}}\mathds{E}_{\rho,\alpha,\omega;a+}^{0;\psi}f(x)=
{^{\rm{C}}\mathds{D}_{a+}^{\alpha;\psi}}\mathds{I}_{a+}^{\alpha;\psi}f(x)=f(x).
\end{eqnarray*}
\section{Concluding remarks}

In this work, we proved some properties associated with a fractional integral operator
involving the three-parameters Mittag-Leffler function in the kernel with respect to another function.
We proved that a Cauchy problem is equivalent to the Volterra integral equation of the second kind, 
established the solution in its closed-form for that, we used the method of successive approximations.
Finally, the inverse operators of the $\mathds{E}_{\rho,\mu,\omega;a+}^{\gamma;\psi}$
and $\mathds{E}_{\rho,\mu,\omega;b-}^{\gamma;\psi}$ were defined.

\section*{Acknowledgment}
The author is grateful to prof. E. Capelas de Oliveira for useful and fruitful discussions.



\end{document}